\documentclass[a4paper]{amsart}

\usepackage{hyperref}
\usepackage{color}
\usepackage{graphicx}

\DeclareMathOperator{\Jac}{Jac}
\DeclareMathOperator{\Sym}{Sym}
\DeclareMathOperator{\ord}{ord}
\DeclareMathOperator{\Res}{Res}

\newtheorem{thm}{Theorem}

\theoremstyle{definition}
\newtheorem{example}{Example}
\newtheorem{lem}{Lemma}

\theoremstyle{definition}
\newtheorem{defn}{Definition}

\begin{document}

\title{Algebraic Solution of Jacobi Inverse Problem and Explicit Addition Laws on Jacobians}
\author{Yaacov Kopeliovich}
\address{Department of Finance, School of Business, University of Connecticut, Storrs}
\email{yaacov.kopeliovich@uconn.edu}
\date{\today}

\begin{abstract}
We formulate a solution to the Algebraic version of the Inverse Jacobi problem. Using this solution we produce explicit addition laws on any algebraic curve generalizing the law suggested by Leykin \cite{BL} in the case of $(n,s)$ curves. This gives a positive answer to a question asked by T.\,Shaska whether addition laws appearing in \cite{BL} can be produced in a coordinate free manner. 
\end{abstract}
\maketitle
\section{Introduction}
Since 1980's Elliptic Curve Cryptography grown into a vast field. In the heart of these cryptography applications is the fact that elliptic curves form an Abelian group. That is if $\mathcal{E}$ is an elliptic curve and $(x_1,y_1)$ and $(x_2,y_2)$ are 2 points on the curve, there is an explicit addition law that gives us the third point on $\mathcal{E}$. In fact, a more general statement holds and for any Abelian group one can devise a cryptographic system similar the systems produced by $\mathcal{E}$. 
This fact led to a search for other examples of Abelian group. One such example is the Jacobian $\Jac(X)$ of any curve $X$. While there are security challenges designing cryptographic systems for curves of high genus there is still a natural question whether explicit addition laws can be formulated for $\Jac(X).$ 
As far as we know there is no easy formulation for such laws. For $g=2$ explicit addition laws were found by Gaudry in \cite{Ga} and for general curves one has the algorithm attributed to Florian Hess \cite{H} and Makdisi \cite{Ma}. 
But these algorithms are not  as straightforward as the algorithms in $g=1,2.$ The one exception is subclass of curves given by an equation: $y^n=x^s+p(x,y)$ where $\deg_yp(x,y)<n$ and $\deg_xp(x,y)<s.$ For these curves the authors of \cite{BL} formulated explicit addition laws and used them to generalize the theory of $\sigma$ and $\wp$ functions from $g=1.$ See \cite{BeLe} for an explicit solution of the Jacobi inverse  problem for $(n,s)$ curves. 
During the Nato conference\footnote{NATO Science for Peace and Security, Advanced Research Workshops Isogeny based post-quantum cryptography Hebrew University of Jerusalem, July 29-31, 2024} 
Tony Shaska raised the question whether these explicit laws can be formulated in coordinate free and equation free way. A possible application for this to cryptography is to build cryptographic systems (for example in genus two) 
given by intersection of algebraic varieties.  
Another possible application is to look for explicit isogenies which require explicit addition laws. 
Our goal in this small note is to answer Shaska's question positively (at least for non-special divisors). We will adept \cite{BL} to  solve the  algebraic Jacobi inverse problem and use it to produce explicit addition laws which in our opinion are simpler that the laws produced by Hess and Makdisi. We work above $\mathbb{C}$ though the construction can be probably performed over any field. The structure of this note is as follows: in section 1 we will formulate and solve the alegbraic Jacobi inverse problem. In section 2 we apply the results from section 1 to obtain addition laws. 

\section{Algebraic Jacobi Inverse Problem and its solution}
To formulate explicit laws we first solve the inverse Jacobi problem. Let 
$W_g=\Sym(X^g)$ be a symmetric product of $X$ $g$ times. We have: 
\begin{thm}
Let $\Jac(X)$ be the Jacobian of $X$, then the natural map: $W_g \to \Jac(X)$ is onto. 
\end{thm}
The proof of this theorem in \cite{FK} is produced through the analytical model of $\Jac(X)$. Mimicking Leykin's construction from \cite{BL}, we show that this theorem can be shown algebraically. Let $w$ be a Weierstrass point on $X$. 
\begin{defn} 
A divisor of degree $0$ is called Mumford divisor if its of the form: $\sum_{i=1}^{g+m}P_i-(g+m)w.$ It follows that it is enough to formulate the addition law for Mumford divisors (as every other divisor of degree $0$ can be presented as a difference of Mumford divisors).
\footnote{This is a direct generalization of Mumford's definition for Hyper-elliptic curves in \cite{Mu}.} 
\end{defn} 
Consider Mumford divisors of the form: $P_1 + \cdots + P_{g+m}-(g+m)w.$ We show: 
\begin{thm}\label{main}
Assume that $D_{g+m}=P_1+ \cdots + P_{g+m}-(g+m)w$ is a non-special divisor in $\Jac(X).$ 
Then there a divisor $D_g=Q_1+ \cdots +Q_g-gw$ such that $D_{g+m}-D_g\equiv 0$ inside $\Jac(X)$.
\end{thm}
\begin{proof}
We first show the following lemma:  
\begin{lem}
Assume that $D_{g+m}=P_1+ \cdots + P_{g+m}-(g+m)w$ is a non-special divisor in $\Jac(X)$. 
Then there a divisor $E_g=Q_1 + \cdots + Q_g-gw$ such that $D_{g+m}+E_g \equiv 0$ inside $\Jac(X)$.
\end{lem}
\begin{proof}
For a divisor $D_{g+m}=\sum_{i=1}^{g+m}P_i$ consider the first $g+m$ functions of the base with unique pole in $w,$ $f_{1}, \ldots, f_{g+m}.$ These exist because $w$ is a Weierstrass point.  
Per discussion in \cite{FK} the leading order of this pole is $g+m$ in $w.$ Define the analogue of Vandermonde determinant: 
\[
  A_{g+m+1\times g+m+1} =
  \left| {\begin{array}{ccccc}
    1 & f_{1}(P) & \cdots & f_{g+m}(P)\\
    1 & f_{1}(P_1) & \cdots & f_{g+m}(P_1)\\
    1 & f_{1}(P_2) & \cdots & f_{g+m}(P_2)\\
    \vdots & \vdots & \ddots & \vdots\\
    1 & f_{1}(P_{g+m}) & \cdots & f_{g+m}(P_{g+m})\\
  \end{array} } \right|
\]
$P$ is an arbitrary point on $X.$ 
Note that this determinant is a well defined function (at least when the divisor is non-special). Because of the determinant properties we have that $\det(A_{g+m+1\times g+m+1})(P_i)=0$ as we have two identical rows in the matrix. On the other hand, the leading order pole in this determinant is precisely $2g+m$ according to the discussion in \cite{FK}. 
Hence, we must have $g$ more zeros for $\det(A_{g+m+1\times g+m+1})=0$. 
Call these zeros $Q_1$, \ldots, $Q_g.$ By definition of the Jacobian we have: 
\begin{equation}
\sum_{i=1}^{g+m}P_i-(g+m)w+\sum_{j=1}^g Q_j - gw \equiv 0,
\end{equation}
as there is a function whose divisor is: $\sum_{i=1}^{g+m}P_i+\sum_{j=1}^gQ_j-(2g+m)w$.
\end{proof}
To complete the proof of the theorem we need to verify that  the divisor $\sum_{j=1}^g Q_i$ can be inverted. 
Repeat the determinant construction:
\[
  B_{g+1\times g+1} =
  \left| {\begin{array}{ccccc}
    1 & f_{1}(Q) & \cdots & f_{g}(Q)\\
    1 & f_{1}(Q_1) & \cdots & f_{g}(Q_1)\\
    1 & f_{1}(Q_2) & \cdots & f_{g}(Q_2)\\
    \vdots & \vdots & \ddots & \vdots\\
    1 & f_{1}(Q_{g}) & \cdots & f_{g}(Q_{g})\\
  \end{array} } \right|.
\]
Accordingly, this polynomial vanishes at the points $Q_1,$ \ldots, $Q_g$ and will have additional $g$ zeros as the function $f_{g}$ has a unique pole of order $2g$ at $w.$ Calling the divisor of the extra zeros: $Q_1'$ ,\ldots, $Q_g'$ we conclude the theorem. 
\end{proof}
\section{Applications}
\subsection{Explicit Addition Law}
For an algebraic curve $X$ be a Riemann surface. We like to formulate the addition law of divisors on 
$\Sym^g(X)$ directly. Note that the formulation of the law in \cite{FK} relies on the analytical construction of the Jacobian. The usual algebraic description of the Jacobian as the class group of $F(X)$ does not produce an explicit addition law. 
We use Theorem~\ref{main} to produce such a law. 
Let $D_g\,{=}\,\sum_{i=1}^gP_i$, $D'_g\,{=}\,\sum_{i=1}^gP_i'$ be two divisors of degree $g$. 
We form a degree $2g$ divisor $\sum_{i=1}^g (P_i+P'_i)$. By Theorem~\ref{main} we have a divisor $E_g=\sum_{i=1}^g Q_i$ such that $\sum_{i=1}^g (P_i+P'_i) -
\sum_{j=1}^gQ_j\equiv 0.$ By definition of $\Jac(X)$ we have: 
\begin{equation}
D_g+D_g'=E_g
\end{equation}
Note that this definition allows you to produce explicit divisors $E_g$ for any divisors of degree $D_k$,
$D_{k_1}$ of degree $k$, $k_1$ respectively such that $k+k_1\geq g.$
\subsection{Multiplication by $n$ and division polynomials}
In this subsection we formulate the multiplication by $n$ law. 
That is we are interested to adopt a similar definition to add divisors of the form : 
$nP_1$, \ldots, $nP_g.$ For this consider a divisor of degree $ng$ and 
select the number of appropriate function $f_1$, \ldots, $f_{gn}.$ Consider the expression: 
\[
  T_{ng+1\times ng+1} =
  \left| {\begin{array}{ccccc}
    1 & f_{1}(P) & \cdots & f_{ng}(P)\\
    1 & f_{1}(P_1) & \cdots & f_{ng}(P_1)\\
    1 & f_{1}(P_2) & \cdots & f_{ng}(P_2)\\
    \vdots & \vdots & \ddots & \vdots\\
    1 & f_{1}(P_g) & \cdots & f_{ng}(P_g)\\
    1 & f'_{1}(P_1) & \cdots & f'_{ng}(P_1)\\
    \vdots & \vdots & \ddots & \vdots\\
    1 & f'_{1}(P_g) & \cdots & f'_{ng}(P_g)\\
    1 & f''_1(P_1) &  \cdots & f''_{ng}(P_1)\\
    \vdots & \vdots & \ddots & \vdots\\
    1 & f^{n-1}_{1}(P_g) & \cdots & f^{n-1}_{ng}(P_1)\\
  \end{array} } \right|
\]
where $f_i^j$ is the $j$-th  derivative of $f_i$ with respect to a local coordinate on $X-w$.
Once again the function above has degree $(n+1)g$ and the order of each $P_i$ appearing in the divisor is precisely $n.$
\subsection{Division Polynomials}
The previous section enables us to formulate division polynomials like condition for torsion points of order $n$ on $\Jac(X).$ Recall the $P_1+ \cdots + P_g-gw $ is a torsion point iff $\sum_{i=1}^g nP_i-ngw\equiv 0$ in $\Jac(X)$. Rewrite this condition as $\sum_{j=1}^{g}(n-1)P_i-(n-1)gw+\sum_{i=1}^gP_i-gw.$ The immediate consequence of this is the following: 
\begin{lem}
Let $T_{(n-1)g+1\times (n-1)g+1}$ be the determinant formed above. Then $P_1+\cdots + P_g$
 is a torsion point of order $n$ if and only if the remaining $g$ zeros of $T_{(n-1)g+1\times (n-1)g+1}$ 
 are $P_1$, \ldots, $P_g$.
\end{lem}
\section{$(n,s)$ curves}
Let us consider $(n,s)$ worked out by Leykin and his co-authors \cite{BL}.  
\begin{defn}
A curve is called $(n,s)$ if it satisfies the equation of the form: $y^n=x^s+p(x,y)$ where $\deg_x p(x,y)<s$ and $\deg_y p(x,y)<n.$
\end{defn}
Assuming $(n,s)=1$ $\infty$ is a Weierstrass point the genus is $g(X)=\frac{(n-1)(s-1)}{2}$
we have that $\ord_{\infty}(x)=n$ and $\ord_{\infty}(y)=s.$ The basis for Weierstrass functions are of the form: $\{x^iy^j| 0\leq i\leq s-1,0\leq j\leq n-1\}.$ Then if $P_i=(x_i,y_i)$ we have: 
\[
  A^{n,s}_{g+m+1\times g+m+1} =
  \left| {\begin{array}{ccccc}
    1 & x & \cdots & x^ly^b\\
    1 & x_1 & \cdots & x_1^ly_1^b\\
    1 & x_2 & \cdots & x_2^ly_2^b\\
    \vdots & \vdots & \ddots & \vdots\\
    1 & x_{g+m} & \cdots & x_{g+m}^ly_{g+m}^b\\
  \end{array} } \right|
\]
Where $l,b$ are powers that are determined so that $\ord_{\infty} x^ly^b=g+m$.
\begin{example}   
Consider for example $(3,s)$ curves given by equation: 
\begin{equation}
y^3=x^s+...
\end{equation}
then $\infty$ is a Weirstrass point. $\ord_{\infty}(x)=3$, $\ord_{\infty}(y)=4$. The base of the functions at this point are:  
$1,x,y,x^2,xy,y^2...$ the gaps are $1,2,5$ and after $6=2g$ you do not have any gaps. Assume that we have a divisor of degree $4:$
$P_1+P_2+P_3+P_4$. Let $P_i=(x_i,y_i).$ Then:
\[
  A^{3,4}_{5\times 5} =
  \left| {\begin{array}{ccccc}
    1 & x & y & x^2 & xy\\
    1 & x_1 & y_1 & x_1^2 & x_1y_1\\
    1 & x_2 & y_2 & x_2^2 & x_2y_2\\
    1 & x_3 & y_3 & x_3^2 & x_3y_3\\
    1 & x_4 & y_4 & x_4^2 & x_4y_4\\    
  \end{array} } \right|
\]
The order of this determinant is $7$ hence you have $3$ more zeros $Q_1$, $Q_2$, $Q_3.$ Now invert $Q_1+Q_2+Q_3$ through another matrix using the functions: $1,x,y,xy.$ This gives a divisor of the form $P_1'+P_2'+P_3'-3\infty$ 
which is equivalent to our \textcolor{red}{original divisor $\Jac(X)$.}
In this case we can say a little more about $Q_1+Q_2+Q_3$, the zeros at hand. 
Let $Q_i=(q_i,q_i')$, and  $f(x,y)=0$ be the equation of the curve. We have the following easy lemma: 
\begin{lem}
Let $p(x)=\prod_{i=1}^3(x-x_i)$ then, 
\begin{equation}
q(x)=\frac{\Res_y(A^{3,4}_{5\times 5},f(x,y))}{p(x)}
\end{equation}
\end{lem}
Furthermore as $y$ is linear in $A^{3,4}_{5\times 5}$ we can easily find $q_i'$ once we know $q_i.$
\end{example} 
 \section{The connection with trace function and Zemel coordinates} 
The connection of the preceeding construction to trace functions becomes clear if you like to consider the problem for divisors of the kind: 
$\sum_{i=1}^k \alpha_i P_i$. In this case the Vandermonde matrix collapses to a matrix 
where $f_{1w}(P_2)$, \ldots, $f_{g+mw}(P_2)$ collapses into something like: $f'_{1w}(P_1)$, \ldots, $f'_{g+mw}(P_1)$ etc. This produces another example of corollary 31 in \cite{Z}
\section{conclusion}
In this note we have formulated solution to the \textcolor{red}{algebraic Jacobi  problem} for non-special divisors on any Riemann surface $X.$ While our construction was performed over $\mathbb{C}$ it can probably be extended to any field. In this case Riemann surfaces will be replaced by algebraic curves. The example for $(3,4)$ curves serves as a basis for explicit addition algorithms on $(3,s)$ curves worked out in \textcolor{red}{[BeK]}. It is an interesting question whether such explicit addition laws can be worked in an abstract fashion.

\end{document}